\documentclass{article}

\usepackage{amsthm}
\usepackage{amsmath}
\usepackage{amssymb}

\newtheorem{theorem}{Theorem}[section]
\newtheorem{conjecture}[theorem]{Conjecture}

\newtheorem{lemma}[theorem]{Lemma}

\newtheorem{corollary}[theorem]{Corollary}

\theoremstyle{definition}
\newtheorem{definition}[theorem]{Definition}

\title{Successive vertex orderings of fully regular graphs}
\author{Lixing Fang \thanks{Institute for Interdisciplinary Information Sciences, Tsinghua University, Beijing, China. Email: 2801511863@qq.com.} \and Hao Huang \thanks{Department of Mathematics, National University of Singapore. Email: huanghao@nus.edu.sg. Research supported in part by a start-up grant at NUS and an MOE Academic Research Fund (AcRF) Tier 1 grant.} \and J\'anos Pach \thanks{R\'enyi Institute, Budapest and IST Austria. Research partially supported by National Research, Development and Innovation Office (NKFIH) grant K-131529 and ERC Advanced Grant ``GeoScape.'' Email: pach@cims.nyu.edu.} \and G\'abor Tardos \thanks{R\'enyi Institute, Budapest. Research partially supported by National Research, Development and Innovation Office (NKFIH) grants K-132696,
SSN-135643, and ERC Advanced Grant ``GeoScape.'' Email: tardos@renyi.hu. } 
\and Junchi Zuo \thanks{Qiuzhen College, Tsinghua University, Beijing, China. Email: zjczzzjc@126.com.}}
\date{}
\begin{document}
\maketitle
\begin{abstract}
A graph $G=(V,E)$ is called {\em fully regular} if for every independent set $I\subset V$, the number of vertices in $V\setminus I$ that are not connected to any element of $I$ depends only on the size of $I$. A linear ordering of the vertices of $G$ is called \emph{successive} if for every $i$, the first $i$ vertices induce a connected subgraph of $G$. We give an explicit formula for the number of successive vertex orderings of a fully regular graph.

As an application of our results, we give alternative proofs of two theorems of Stanley and Gao \& Peng, determining the number of linear \emph{edge} orderings of complete graphs and complete bipartite graphs, respectively, with the property that the first $i$ edges induce a connected subgraph.

As another application, we give a simple product formula for the number of linear orderings of the hyperedges of a complete 3-partite 3-uniform hypergraph such that, for every $i$, the first $i$ hyperedges induce a connected subgraph. We found similar formulas for complete (non-partite) 3-uniform hypergraphs and in another closely related case, but we managed to verify them only when the number of vertices is small.

\end{abstract}

\section{Introduction}

In preparation for a computing contest, the first-named author bumped into the following question. In how many different ways can we arrange the first $mn$ positive integers in an $m\times n$ matrix so that for each entry $i$ different from $1$, there is a smaller entry either in the same row or in the same column? After some computation, he accidentally found the formula
$$(mn)!\cdot\frac{m+n}{\binom{m+n}{m}}$$
for this quantity, which he was able to verify by computer up to $m,n\le 2000$. It turns out that at about the same time, the same question was asked by S. Palcoux on MathOverflow~\cite{Pa18}, which has led to interesting results by Stanley \cite{Stanley} and by Gao and Peng \cite{GaoPeng}. We also posed the question as Problem 4 at the 2019 Mikl\'os Schweitzer Memorial Competition in Hungary, see~\cite{Sch19}.
\smallskip

Many outstanding mathematicians contemplated what makes a mathematical formula beautiful. One of the often proposed criteria was that, even if we somehow hit upon it, there is no easy way to verify it; see, e.g., ~\cite{Tu77}. The above formula seems to meet this criterion.
\smallskip

First, we reformulate the above question in graph-theoretic terms. A \emph{shelling} of a graph $G$ (regarded as a 1-dimensional simplicial complex) is a linear ordering of its edges such that, for every $i$, the first $i$ edges induce a connected subgraph in $G$. Clearly, the number of different ways to enumerate the $mn$ positions of an $m\times n$ matrix with the required properties is equal to the number of shellings of $K_{m,n}$, a complete bipartite graph with $m$ and $n$ vertices in its classes. Stanley and Gao and Peng were the first to establish the following formulas.

\begin{theorem}\label{thm1}
{\bf (i)} {\rm (Stanley, \cite{Stanley})}
The number of shellings of the complete graph $K_n$ on $n\ge2$ vertices is
$$\binom{n}{2}!\cdot \frac{n!}{2 \cdot (2n-3)!!}$$

{\bf (ii)} {\rm(Gao-Peng~\cite{GaoPeng})}
The number of shellings of the complete bipartite graph $K_{m,n}$ with $m\ge1$ and $n\ge 1$ vertices in its classes is
$$(mn)! \cdot \frac{m+n}{\binom{m+n}{m}}.$$
\end{theorem}

The aim of the present note is to approach the above problem from a slightly different angle, by counting \emph{vertex orders} rather than edge orders.

\begin{definition}
Let $G$ be a graph with vertex set $V(G)$. A \emph{linear ordering} $\pi: V(G)\rightarrow \{1,2,\ldots,|V(G)|\}$ of $V(G)$ is said to be \emph{successive} if, for every $i\ge1$, the subgraph of $G$ induced by the vertices $v\in V(G)$ with $\pi(v)\le i$ is connected.
\end{definition}

Equivalently, $\pi$ is a successive vertex ordering if and only if for every vertex $v\in V(G)$ with $\pi(v)>1$, there is an adjacent vertex $v'\in V(G)$ with $\pi(v')<\pi(v)$.
\smallskip

Let $\sigma(G)$ denote the number of successive linear orderings of $V(G)$. In a probabilistic framework, it is often more convenient to calculate the probability $\sigma'(G)$ that a randomly and uniformly chosen linear ordering of $V(G)$ is successive. Obviously, we have $\sigma'(G)=\sigma(G)/|V(G)|!$ For an arbitrary graph $G$, usually it is hopelessly difficult to determine these parameters. We need to restrict our attention to some special classes of graphs.

A set of vertices $I\subseteq V(G)$ is \emph{independent} if no two elements of $I$ are adjacent. The size of the largest independent set in $G$ is denoted by $\alpha(G)$.

\begin{definition} A graph $G$ is called \emph{fully regular} if for an independent set $I\subseteq V(G)$, the number of vertices in $V(G)\setminus I$ not adjacent to any element of $I$ is determined by the size of $I$.
\end{definition}

Clearly, a graph $G$ is fully regular if there exist numbers $a_0, a_1,\ldots, a_{\alpha(G)}$ such that for any independent set $I\subseteq V(G)$, the number of vertices in $V(G)\setminus I$ not adjacent to any element of $I$ is $a_{|I|}$. We call the numbers $a_i$ the \emph{parameters} of the fully regular graph $G$. We must have $a_0=|V(G)|$ and $a_{\alpha(G)}=0$.
\smallskip

In Section~\ref{sec2}, we use the inclusion-exclusion principle to prove the following formula for the number of successive orderings of a fully regular graph.

\begin{theorem}\label{main}
Let $G$ be a fully regular graph with parameters $a_0,a_1,\dots,a_\alpha$, where $\alpha=\alpha(G)$. We have
$$\sigma'(G)=\sum_{i=0}^{\alpha}\prod_{j=1}^i\frac{-a_j}{a_0-a_j},$$
$$\sigma(G)=a_0!\sum_{i=0}^{\alpha}\prod_{j=1}^i\frac{-a_j}{a_0-a_j}.$$
\end{theorem}

Here and in some other formulas in this paper, we have empty products, such as $\prod_{j=1}^0a_j$. These products should be interpreted as having value $1$.

The terms corresponding to $i=\alpha$ in the sums vanish, because we have $a_\alpha=0$. Thus, the upper limit $\alpha$ in the sums can be replaced by $\alpha-1$.
\smallskip

The \emph{‌line graph} of a hypergraph $H$ is a graph whose vertex set is the set of hyperedges of $H$ and two hyperedges are adjacent if and only if their intersection is nonempty \cite{Berge}, \cite{Bermond}.
\smallskip

It is easy to see that the line graph of every \emph{complete} $r$-uniform hypergraph and every \emph{complete $r$-partite} $r$-uniform hypergraph is fully regular, for any integer $r\ge 2$.

We can generalize these examples as follows. Fix a sequence $d_1,\dots,d_t$ of positive integers, let $d=\sum_{j=1}^td_j$, and let $V_1,\ldots,V_t$ be pairwise disjoint sets.  Consider the $d$-uniform hypergraph $H$ on the vertex set $V=\cup_{j=1}^tV_j$, consisting of all hyperedges $e$ such that $|e\cap V_j|=d_j$ for every $j$. The number of hyperedges of $H$ is $\prod_{j=1}^t\binom{|V_j|}{d_j}$. We claim that the line graph $L(H)$ of $H$ is fully regular. To see this, take an independent set $I$ of size $i$ in $L(H)$. Obviously, all hyperedges of $H$ which correspond to a vertex of $L(H)$ that does not intersect any hyperedge in $I$ form a complete uniform hypergraph on a smaller number of vertices. The number of these hyperedges (vertices of $L(H)$) is $a_i:=\prod_{j=1}^t\binom{|V_j|-id_j}{d_j}$. This number depends only on $i=|I|$, proving that $L(H)$ is fully regular.
\smallskip

The case $d=2$, where $H$ is a \emph{graph} (2-uniform hypergraph) is especially interesting, because a successive vertex ordering of its ‌line graph $L(H)$ is the same as a \emph{shelling} of $H$. Unfortunately, such a direct connection fails to hold for $d>2$.
\smallskip

For $d=2$, we have two possibilities: (i) the case $t=1$, $d_1=2$ yields complete graphs $H=K_n$; (ii) the case $t=2$, $d_1=d_2=1$ yields complete bipartite graphs $H=K_{m,n}$ for some $m$ and $n$. In case (i), we have that $L(H)$ is fully regular with parameters $a_i=\binom{n-2i}2$, for every $0\le i\le\lfloor n/2\rfloor=\alpha(K_n)$. In case (ii), we obtain $a_i=(m-i)(n-i)$ for every $0\le i\le\min(m,n)=\alpha(K_{m,n})$. A direct application of Theorem~\ref{main} to $L(H)$ gives

\begin{corollary}\label{thm1'}
{\bf (i)}
The number of shellings of the complete graph $K_n$ on $n\ge2$ vertices is
$$\binom{n}{2}!\cdot \sum_{i=0}^{\lfloor n/2\rfloor}\prod_{j=1}^i\frac{1}{1-\binom{n}{2}/\binom{n-2j}{2j}}.$$
{\bf (ii)}
The number of shellings of the complete bipartite graph $K_{m,n}$ with $m\ge1$ and $n\ge 1$ vertices in its classes is
$$(mn)! \cdot \sum_{i=0}^{\min(m,n)}\prod_{j=1}^i\frac{1}{1-mn/((m-j)(n-j))}.$$
\end{corollary}

In Section~\ref{sec3}, we prove that the summation formulas in Corollary~\ref{thm1'} are equal to the product formulas in Theorem~\ref{thm1} obtained by Richard Stanley~\cite{Stanley}, Yibo Gao and Junyao Peng \cite{GaoPeng}. Thereby, we provide alternative proofs for the latter results.

It is always interesting when a summation formula can be turned into a nice product formula. If this is possible, it often yields some deeper insights. We were able to turn the summation formula of Theorem~\ref{main} into a product formula in yet another case: applying it to line graphs of complete 3-partite 3-uniform hypergraph. In this case, we have $t=3$ and $d_1=d_2=d_3=1$. In Section~\ref{sec4}, we establish the following result.

\begin{theorem}\label{new}
Let $K_{m,n,p}$ denote the complete 3-partite 3-uniform hypergraph with $m, n,$ and $p$ elements in its vertex classes, and let $G$ denote its line graph. Set $b_i=mn+np+mp-i(m+n+p-i)$. Then the number of successive orderings of the vertices of $G$ is
$$\sigma(G)=\frac{(mnp-1)!\prod_{i=1}^{m+n+p-1}b_i}{\prod_{i=1}^{m-1}b_i\prod_{i=1}^{n-1}b_i\prod_{i=1}^{p-1}b_i}=(mnp)!\cdot\frac{\prod_{i=m}^{m+p}b_i}{mnp\prod_{i=1}^{p-1}b_i},$$
where the fractions should be evaluated disregarding all zero factors in both the numerator and the denominator.
\end{theorem}

We found similar product formulas for the number of successive vertex orderings of the line graph of a complete 3-uniform hypergraph $K_n^{(3)}$ (where $t=1$ and $d_1=3$), and the line graph of $K_{m,n}^{(1,2)}$ (where $t=2$, $d_1=1$, $d_2=2$) but we were unable to verify them. We state them as  conjectures in the last section, together with other open problems and remarks.

\section{Successive vertex orderings\\--Proof of Theorem~\ref{main}}\label{sec2}

In this section, we apply the inclusion-exclusion principle to establish Theorem~\ref{main}.

\begin{proof}
It is enough to prove the first formula.
Consider a uniform random linear ordering $\pi$ of $V(G)$.
For any vertex $v\in V(G)$, let $B_v$ denote the \emph{``bad''} event that $v$ is not the first vertex, but $v$ comes before all vertices adjacent to it. In other words, we have $\pi(v) \neq 1$ and $\pi(v)<\pi(v')$ for every vertex $v'$ adjacent to $v$.  Note that if two vertices, $v$ and $v'$, are adjacent, then $B_v$ and $B_{v'}$ are mutually exclusive events, i.e., we have
$$\mathbb{P}(B_{v} \wedge B_{v'})=0.$$
Indeed, the inequalities $\pi(v)<\pi(v')$ and $\pi(v')<\pi(v)$ cannot hold simultaneously. Therefore, the vertices $v$ for which a bad event occurs always form an independent set $I$. The linear order $\pi$ is successive if and only if this independent set is empty. By the inclusion-exclusion formula, we have
\begin{equation}\label{eq0}
\sigma'(G)=\sum_{i=0}^{\alpha} (-1)^i \cdot \sum_{I: |I|=i} \mathbb{P}(\bigwedge_{v \in I} B_{v}).
\end{equation}
Here, the second sum is taken over all independent sets $I$ of size $i$ in $G$. We also use the convention that empty intersection of events returns the universal event of probability $1$.
\smallskip

For a given independent set $I$, denote by $N(I)$ the \emph{neighborhood} of $I$, that is, the set of vertices either in $I$ or adjacent to at least one vertex that belongs to $I$. Clearly, we have
$$|N(I)|=|V(G)|-a_{|I|}=a_0-a_{|I|}.$$
\smallskip

We start with evaluating the probability of the event $B_I:=\bigwedge_{v \in I} B_{v}$ for an independent set $I$ of size $i$. Let $\rho$ be an enumeration of $I$, that is $I=\{\rho(1),\rho(2),\dots,\rho(i)\}$. Consider first the event $C_\rho$ that $B_I$ happens and we also have $\pi(\rho(1))<\pi(\rho(2))< \cdots <\pi(\rho(i))$.
Clearly, $C_\rho$ occurs if and only if $\pi^{-1}(1)\notin N(I)$ and $\rho(j)$ is minimal among the vertices in $N(\{\rho(j),\rho(j+1),\dots,\rho(i)\})$ for $1\le j\le i$. These $i+1$ events are mutually independent and we clearly have
$$\mathbb P(\pi^{-1}(1)\notin N(I))=\frac{|V(G)\setminus N(I)|}{|V(G)|}=\frac{a_i}{a_0},$$
$$\mathbb P(\rho(j)\hbox{ is minimal in }N(\{\rho(j),\rho(j+1),\dots,\rho(i)\}))$$$$=\frac1{|N(\{\rho(j),\rho(j+1),\dots,\rho(i)\})|}=\frac1{a_0-a_{i-j+1}}.$$
Therefore, we have
$$\mathbb P(C_\rho)=\frac{a_i}{a_0}\prod_{j=1}^i\frac1{a_0-a_j}.$$
Now the event $B_I$ is the disjoint union of the events $C_\rho$ where $\rho$ runs over the $i!$ possible enumerations of $I$, so we have
\begin{equation}\label{u}
\mathbb P(B_I)=i!\frac{a_i}{a_0}\prod_{j=1}^i\frac1{a_0-a_j}.
\end{equation}
\smallskip

As $\mathbb P(B_I)$ does not depend on the independent set $I$ beyond its size $i$, we can avaluate Equation~\ref{eq0} by simply counting the independent sets in $G$ of any given size.  The first vertex $v_1$ of an independent set can be any one of the $|V(G)|=a_0$ vertices. After choosing $v_1,\dots v_j$, the next vertex of an independent set must be outside $N(\{v_1,\dots,v_j\})$, so we have $a_j$ choices. This implies that the number of size $i$ independent sets in $G$ is
$$\frac{\prod_{j=0}^{i-1}a_j}{i!}.$$
We had to divide by $i!$, because the vertices of an independent set can be selected in an arbitrary order. Plugging this formula and Equation~\ref{u} in Equation~\ref{eq0} proves the theorem.
\end{proof}

\section{Shelling the edges of a graph\\
--Alternative proof of Theorem~\ref{thm1}}\label{sec3}



The aim of this section is to give alternative proofs of parts (i) and (ii) of Theorem~\ref{thm1}, by  establishing a common generalization of the two statements; see Theorem~\ref{together} below.
\smallskip

As explained in the Introduction, a shelling of a graph is the same as a successive ordering of the vertices of its ‌line graph. We also noted that the ‌line graphs of complete graphs and complete bipartite graphs are fully regular. Thus, we can apply Theorem~\ref{main} to obtain Corollary~\ref{thm1} for the number of shellings of $K_n$ and $K_{m,n}$. In this way, however, we obtain \emph{summation} formulas, while Theorem~\ref{thm1} gives much nicer \emph{product} formulas with low degree factors.
\smallskip

Two distinct \emph{edges} of a graph $G$ are called \emph{adjacent} if they share a vertex. Otherwise, we call them \emph{independent}. A family of pairwise independent edges is called a \emph{matching}. The \emph{matching number} $\nu(G)$ of $G$ is the size of the largest matching in $G$. Matchings in $G$ correspond to independent sets in the ‌line graph $L(G)$ of $G$, so $\nu(G)=\alpha(L(G))$.

First, we characterize all graphs whose line graphs are fully regular.

\begin{lemma}\label{obs}
The ‌line graph of a graph $G$ is fully regular if and only if every edge of $G$ is adjacent to the same number of other edges and every pair of independent edges are connected by the same number of edges.
\end{lemma}

\begin{proof}
If the ‌line graph of $G$ is fully regular with parameters $a_0,a_1,\dots$, then $a_0=|E(G)|$ and for any edge $e$ of $G$, the number of edges in $G$ distinct from and not adjacent to $e$ is $a_1$. Thus, $e$ is adjacent to $d:=a_0-a_1-1$ edges. If $e$ and $e'$ are independent, then the number of edges distinct from both and not adjacent to either of them is $a_2$. Therefore, we have $a_2=a_0-2-2d+\lambda$, where $\lambda$ is the number of edges connecting $e$ and $e'$ (and, hence, adjacent to both). Thus, we have $\lambda=a_2-a_0+2d+2=a_0-2a_1+a_2$, proving the ``only if'' part of the lemma.
\smallskip

To show the ``if'' part, we assume that every edge is adjacent to $d$ other edges in $G$ and every pair of independent edges is connected by $\lambda$ edges. We need to show that the ‌line graph of $G$ is fully regular, that is, for any independent set $I$ in the ‌line graph (i.e., for any matching $I$ in $G$), the number of vertices in the ‌line graph at distance at least 2 from $I$ is determined by the size $i:=|I|$. We have $i$ vertices at distance zero. We have $d$ adjacent edges for each edge in $I$, yielding $id$ vertices at distance 1, but we have to subtract from this the edges that we counted twice. Note that no edge is counted more than twice, and that in the line graph any pair of edges that belong to $I$ have exactly $\lambda$ common neighbors. So, in the line graph there are exactly $id-\binom i2\lambda$ vertices at distance $1$ from $I$. Hence, $a_i=|E(G)|-i(d+1)+\binom i2\lambda$, which completes the proof.
\end{proof}

The main result of this section is the following.

\begin{theorem}\label{together}
Let $G$ be a graph with matching number $\nu$. Suppose that every edge of $G$ is adjacent to $d$ other edges, and  every pair of independent edges is connected by exactly $\lambda$ edges in $G$.

The number $N$ of shellings of $G$ satisfies
$$\frac N{|E(G)|!}=\frac\nu{\binom{\frac{d+1}{\lambda/2}}{\nu-1}}.$$
\end{theorem}	

In the last theorem, we used binomial coefficients of the form $a\choose b$, where $a$ is not necessarily an integer. They should be interpreted in the usual way: as a polynomial of $a$ with degree $b$.

For $G=K_n$, $n\ge2$ the conditions of the theorem are satisfied with $\nu=\lfloor n/2 \rfloor$, $d=2n-4$, and $\lambda=4$. Therefore, Theorem~\ref{together} implies that the number $N$ of shellings of the complete graph $K_n$ satisfies
$$\frac N{{\binom{n}{2}!}}=\frac{\lfloor n/2\rfloor}{\binom{\frac{2n-3}2}{\lfloor n/2 \rfloor-1}}=\frac{2^{\lfloor n/2\rfloor-1}\lfloor n/2\rfloor!(2n-2\lfloor n/2\rfloor-1)!!}{(2n-3)!!}=\frac{n!}{2 \cdot (2n-3)!!},$$
giving an alternative proof of part (i) of Theorem~\ref{thm1}.
\smallskip

For $G=K_{m, n}$ with $1\le m\le n$, we have $\nu=m$, $d=m+n-2$, and $\lambda=2$. Hence, the number $N$ of shellings of the complete bipartite graph $K_{m,n}$ satisfies
$$\frac N{(mn)!}=\frac m{\binom{m+n-1}{m-1}}=\frac{m+n}{\binom{m+n}{m}},$$
providing an alternative proof of part (ii) of Theorem~\ref{thm1}.
\smallskip

We need the following identity for binomial coefficients.

\begin{lemma}\label{lem_identity}
Let $\alpha$ be a non-negative integer and let $\beta$ and $\gamma$ be reals.
Then we have
$$\sum_{t=0}^\alpha(-1)^t\frac{\binom\alpha t \binom\beta t}{\binom\gamma t} = \frac{\binom{\gamma-\beta}\alpha}{\binom\gamma\alpha},$$
unless $\gamma$  is a non-negative integer smaller than $\alpha$.

In the latter case, neither side of the identity is defined.
\end{lemma}

\begin{proof}
We start with the following simple equations that hold for non-negative integers $t\le\alpha$ and $v$, respectively:
$$\frac{\binom\alpha t}{\binom\gamma t}=\frac{\binom{\gamma-t}{\alpha-t}}{\binom\gamma\alpha},$$
$$\binom uv=(-1)^v\binom{v-u-1}v.$$
Using them we obtain
$$(-1)^t\frac{\binom\alpha t \binom\beta t}{\binom\gamma t}=(-1)^t\frac{\binom{\gamma-t}{\alpha-t}\binom\beta t}{\binom\gamma\alpha}=\frac{(-1)^\alpha}{\binom\gamma\alpha}\binom{\alpha-\gamma-1}{\alpha-t}\binom\beta t.$$
Using Vandermonde's identity we obtain
$$\sum_{t=0}^\alpha(-1)^t\frac{\binom\alpha t \binom\beta t}{\binom\gamma t}=\frac{(-1)^\alpha}{\binom\gamma\alpha}\sum_{t=0}^\alpha\binom{\alpha-\gamma-1}{\alpha-t}\binom\beta t=\frac{(-1)^\alpha\binom{\alpha+\beta-\gamma-1}\alpha}{\binom\gamma\alpha}=\frac{\binom{\gamma-\beta}\alpha}{\binom\gamma\alpha}.$$
\end{proof}

Now we are ready to present the proof of Theorem~\ref{together}.

\begin{proof}[Proof of Theorem~\ref{together}:]
Let $L(G)$ denote the ‌line graph of $G$. The independence number $\alpha(L(G))$ is equal to the matching number $\nu$ of $G$. By Lemma~\ref{obs}, $L(G)$ is fully regular. Its associated parameters were calculated as $a_j=|E(G)|-j(d+1)+\binom j2\lambda$, for $0\le j\le\nu$. Note that $0=a_\nu=|E(G)|-\nu(d+1)+\binom\nu2\lambda$ and, hence, $a_0=|E(G)|=\nu(d+1)-\binom\nu2\lambda$. The number $N$ of shellings of $G$ is  equal to $\sigma(L(G))$. Therefore,
\begin{eqnarray*}
\frac N{|E(G)|!}=\sigma'(L(G))&=&\sum_{i=0}^{\nu}\prod_{j=1}^i\frac{-a_j}{a_0-a_j}\\
&=&\sum_{i=0}^{\nu-1}\prod_{j=1}^i\frac{(j(d+1)-\binom j2\lambda)-(\nu(d+1)-\binom\nu2\lambda)}{j(d+1)-\binom j2\lambda}\\
&=&\sum_{i=0}^{\nu-1}(-1)^i\prod_{j=1}^i\frac{(\nu-j)(d+1-(\nu+j-1)\lambda/2)}{j(d+1-(j-1)\lambda/2)}\\
&=&\sum_{i=0}^{\nu-1}(-1)^i\prod_{j=1}^i\frac{(\nu-j)(\frac{d+1}{\lambda/2}-\nu+1-j)}{j(\frac{d+1}{\lambda/2}+1-j)}\\
&=&\sum_{i=0}^{\nu-1}(-1)^i\frac{\binom{\nu-1}i\binom{\frac{d+1}{\lambda/2}-\nu}i}{\binom{\frac{d+1}{\lambda/2}}i}\\
&=&\frac{\binom{\nu}{\nu-1}}{\binom{\frac{d+1}{\lambda/2}}{\nu-1}}=\frac\nu{\binom{\frac{d+1}{\lambda/2}}{\nu-1}},
\end{eqnarray*}
where the first line comes from Theorem~\ref{main}. The second line was obtained by substituting the values of $a_j$ as computed above and removing the vanishing summand for $i=\nu$. We used Lemma~\ref{lem_identity} to obtain the last line. This completes the proof of Theorem~\ref{together}.
\end{proof}

\section{Complete 3-uniform 3-partite hypergraphs\\
--Proof of Theorem~\ref{new}}\label{sec4}

For any positive integers $m$, $n$, and $p$, let $K_{m,n,p}$ denote the complete 3-uniform 3-partite hypergraph with $m$, $n$, and $p$ vertices in its three vertex classes, and let $G=G_{m,n,p}$ denote the line graph of $K_{m,n,p}$. As we have seen in the Introduction, $G$ is a fully regular graph with parameters $a_i=(m-i)(n-i)(p-i)$.
Theorem~\ref{main} gives a summation formula for $\sigma(G)$, the number of successive  orderings of the vertices of $G$. In this section, we prove Theorem~\ref{new}, which is a nice product formula for the same quantity.
\smallskip

We need some notation. The parameters $a_i$ of the fully regular graph $G$ are defined \emph{a priory} only for $i\le\alpha(G)=\min(m,n,p)$. However, here we define
$$a_i:=(m-i)(n-i)(p-i),$$
for all $i$. Notice that the parameters $b_i$ defined in the theorem satisfy
$$b_i=\frac{a_0-a_i}i.$$
For fixed $i$, $a_i$ and $b_i$ are multilinear polynomials of the parameters $m$, $n$, and $p$.
\smallskip


In Theorem~\ref{new}, zero factors show up if $b_i=0$ for some $i$, which may occasionally happen (for instance, in the case $m=8$, $n=p=2$, we have $b_6=0$.) Looking at the definition  of $b_i$, one can see that this happens for at most two values of $i$ (symmetric about $(m+n+p)/2$), and they must be smaller than $\max(m,n,p)$, but larger than the sum of the other two parameters among $m,$, $n$, and $p$. This is a nuisance, which can be avoided by using the second formula in the theorem and choosing $p$ not to be the single largest of the three parameters. Note also, that when zero factors do show up in one or both of the fractions in the theorem, then they have the same number of them (one or two) in both the numerator and the denominator.
\smallskip

We start with a polynomial equality.

\begin{lemma}\label{poly}
For positive integers $m$, $n$, and $p$ and the numbers $a_j$, $b_j$ depending on them, we have
$$\sum_{i=0}^{p-1}\frac{(-1)^i}{i!}\prod_{j=1}^ia_j\prod_{j=i+1}^{p-1}b_j=p\prod_{j=m+1}^{m+p-1}b_j.$$
\end{lemma}

\begin{proof}
Let $P(m,n,p)$ and $Q(m,n,p)$ denote the left-hand side and the right-hand side, resp., of the equation to be verified. Recall that $a_j$ and $b_j$ are multilinear polynomials of $m$, $n$, and $p$. Thus, for a fixed $p$, both $P$ and $Q$ are polynomials in $m$ and $n$, whose degree is at most $p-1$ in either variable.
\smallskip

We prove the equality $P(m,n,p)=Q(m,n,p)$ by induction on $p$. We assume that the equation holds for any positive integer less than $p$ and for every $m$ and $n$, and we will prove that it also holds for $p$. For fixed $m$ and $p$, both $P(m,n,p)$ and $Q(m,n,p)$ are polynomials in $n$ of degree less than $p$. So, it is enough to find $p$ distinct values of $n$, where these polynomials agree. We will do that for $n=0,1,\dots,p-1$.
\smallskip

In case $n=0$, we have $a_0=mnp=0$ and $b_j=(a_0-a_j)/j=-a_j/j$. For any $0\le i\le p-1$, we have
$$\frac{(-1)^i}{i!}\prod_{j=1}^ia_j\prod_{j=i+1}^{p-1}b_j=\prod_{j=1}^ib_j\prod_{j=i+1}^{p-1}b_j=\prod_{j=1}^{p-1}b_j.$$
Therefore,
$$P(m,0,p)=\sum_{i=0}^{p-1}\prod_{j=1}^{p-1}b_j=p\prod_{j=1}^{p-1}b_j=p\prod_{j=m+1}^{m+p-1}b_j=Q(m,0,p),$$
as required. Here we used that $b_j=b_{m+n+p-j}$, which is clear from our formula for $b_j$.
\smallskip

Suppose now that $1\le n<p$. We have $a_n=0$, so all terms in $P(m,n,p)$ corresponding to $i\ge n$ vanish. We can collect $\prod_{j=n}^{p-1}b_j$ from the remaining terms to obtain
$$P(m,n,p)=\prod_{j=n}^{p-1}b_j\sum_{i=0}^{n-1}\prod_{j=1}^ia_j\prod_{j=i+1}^{n-1}b_j=\prod_{j=n}^{p-1}b_jP(m,p,n),$$
where we used that permuting the parameters $m$, $n$, and $p$ (in this case, switching the roles of $n$ and $p$) has no effect on the values $a_j$ and $b_j$.
\smallskip

We have $P(m,p,n)=Q(m,p,n)$, by the induction hypothesis. This yields
\begin{eqnarray*}
P(m,n,p)&=&Q(m,p,n)\prod_{j=n}^{p-1}b_j\\
&=&n\prod_{j={m+1}}^{m+n-1}b_j\prod_{j=m+n+1}^{m+p}b_j\\
&=&n\frac{b_{m+p}}{b_{m+n}}\prod_{j=m+1}^{m+p-1}b_j\\
&=&p\prod_{j=m+1}^{m+p-1}b_j=Q(m,n,p).
\end{eqnarray*}
In the second line, we used that $b_j=b_{m+n+p-j}$ and in the last line we used the facts $b_{m+p}=mp$ and $b_{m+n}=mn.$

Since we found $p$ distinct values of $n$, for which $P(m,n,p)=Q(m,n,p)$, the two polynomials must agree for all $n$. This completes the induction step and the proof of the lemma.
\end{proof}

\begin{proof}[Proof of Theorem~\ref{new}]
Notice that the two expressions for $\sigma(G)$  in the theorem are equal. Moreover, the second fraction can be obtained from the first by cancelling equal terms. This is immediate using the symmetry $b_i=b_{m+n+p-i}$, which implies $\prod_{i=1}^{n-1}b_i=\prod_{i=m+p+1}^{m+n+p-1}b_i$. It remains to prove that $\sigma'(G_{m,n,p})=\prod_{i=m}^{m+p}b_i/(mnp\prod_{i=1}^{p-1}b_i)$ as implied by the second expression in the theorem. Using the symmetry of the first expression, we can assume without loss of generality that $p\le m$ and $p\le n$. With this assumption, there are no zero factors in the fraction to worry about, and we know that the independence number of $G$ is $p$.
\smallskip

By Theorem~\ref{main}, we have
\begin{eqnarray*}
\sigma'(G)&=&\sum_{i=0}^p\prod_{j=1}^i\frac{-a_j}{a_0-a_j}\\
&=&\sum_{i=0}^{p-1}\frac{(-1)^i}{i!}\prod_{j=1}^i\frac{a_j}{b_j}\\
&=&\frac{\sum_{i=0}^{p-1}\frac{(-1)^i}{i!}\prod_{j=1}^ia_j\prod_{j=i+1}^{p-1}b_j}{\prod_{j=1}^{p-1}b_j}.
\end{eqnarray*}
By Lemma~\ref{poly}, the numerator of the last expression can be written as $p\prod_{j=m+1}^{m+p-1}b_j$. Substituting $b_m=np$ and $b_{m+p}=mp$, the theorem follows.
\end{proof}

\section{Comments and open problems}

\noindent{\bf A.} Theorem~\ref{together} appears to be more general than the its special cases, the two parts of Theorem~\ref{thm1}. However, it is not hard to show by a case analysis that the only connected graphs with fully regular ‌line graphs are complete graphs, complete bipartite graphs, and the cycle $C_5$.
\smallskip

\noindent{\bf B.} Every $d$-uniform hypergraph $H$ (more precisely, the closure of its hyperedges under containment) can be regarded as a $(d-1)$-dimensional simplicial complex.
An enumeration $E_1,E_2,\dots, E_n$ of the hyperedges of $H$ is called a \emph{shelling} if, for every $1\le i<j\le n$, there exists $1\le k<j$ with $E_i\cap E_j\subseteq E_k\cap E_j$ and $|E_k\cap E_j|=d-1$. See \cite{Ziegler}. In the case $d=3$, this means that for every $E_j$ with $j>1$, (i) there is a hyperedge $E_k$ preceding it which meets $E_j$ in two points, and (ii) either $E_j\not\subseteq\cup_{i<j}E_i$ or there are two preceding hyperedges that meet $E_j$ in distinct point pairs. If only condition (i) is satisfied, then the ordering is called a \emph{weak shelling}.

As we have remarked in the Introduction, for $d>2$ it is not true that every successive vertex ordering of the line graph of a $d$-uniform hypergraph $H$ corresponds to a shelling of $H$. Nevertheless, every complete $d$-uniform hypergraph and every complete $d$-uniform $d$-partite hypergraph admits a shelling. It would be interesting to compute the number of shellings and the number of weak shellings of these hypergraphs, but Theorem~\ref{main} is not applicable here. Unfortunately, we have no evidence that either of these questions has an answer that can be expressed by a nice-looking summation or product formula.

Observe that every weak shelling of the complete 3-uniform 3-partite hypergraph $K_{m,n,p}$ with $m, n,$ and $p$ vertices in its classes corresponds to an arrangement of the first $mnp$ positive integers in the 3-dimensional $m\times n\times p$ matrix such that for each entry larger than $1$, there is a smaller entry in one of the 3 \emph{rows} passing through it, parallel to a side of the matrix. Therefore, counting the number of weak shellings in this case can be regarded as another natural 3-dimensional generalization of the question described at the beginning of the Introduction, which is different from the generalization given in Theorem~\ref{new}. (Note that the successive vertex orderings of the line graph of $K_{m,n,p}$, which were counted by Theorem~\ref{new}, correspond to arrangements of the first $mnp$ integers in the same matrix such that for each entry larger than $1$, there is a smaller entry in one of the 3 coordinate \emph{planes} passing through it.)

\smallskip

\noindent{\bf C.} Let $K_n^{(3)}$ stand for the complete 3-uniform hypergraph on $n$ vertices. We saw that the line graph $L(K_n^{(3)})$ is fully regular with parameters $\alpha=\lfloor n/3\rfloor$ and $a_i=\binom{n-3i}3$. Therefore, Theorem~\ref{main} gives a summation formula for $\sigma'(L(K_n^{(3)}))$. In the conjecture below, we propose a nice product formula instead.

\begin{conjecture}
$$\sigma'(L(K_n^{(3)}))=\left\lfloor\frac n3\right\rfloor\cdot\frac{\prod\limits_{k=n+1\atop3\nmid k}^{n+\lfloor n/2\rfloor-2}c_k}{\prod\limits_{k=3\atop3\mid k}^{n-3}c_k},$$
where
$$c_k=6\frac{\binom n3-\binom{n-k}3}k=3n^2-6n+2-k(3n-3-k).$$
\end{conjecture}

We verified this conjecture for all $n\le100$. It would be interesting to find similar product formulas for the number of successive vertex orderings of the line graphs of complete $d$-uniform or complete $d$-uniform $d$-partite hypergraphs also for $d>3$. 

\smallskip

\noindent {\bf D.} Consider the $3$-uniform ``bipartite'' hypergraph $K_{m, n}^{(1,2)}$ on the vertex set $V = V_1 \cup V_2$ with $|V_1|=m$, $|V_2|=n$, consisting of all subsets $e\subseteq V$ such that $|e \cap V_1| =1$ and $|e \cap V_2|=2$. We showed in the Introduction that its line graph is fully regular with parameters $\alpha=\min\{m, n/2\}$ and $a_i=(m-i)\binom{n-2i}{2}$. Therefore, Theorem~\ref{main} gives a summation formula for $\sigma'(L(K_{m,n}^{(1,2)}))$. As in the case $K_n^{(3)}$, we conjecture that the following product formula holds.
\begin{conjecture}
Let $d_i=(a_0-a_i)/i$, then
$$\sigma'(L(K_{m,n}^{(1,2}))=m \cdot \prod_{i=1}^{m-1} \frac{mn-\binom{m+1}{2}+\binom{i}{2}}{d_i},$$
where the fractions should be evaluated disregarding all zero factors in both the
numerator and the denominator.
\end{conjecture}
We verified this conjecture for $m, n \le 50$.


\begin{thebibliography}{99}
\bibitem{Berge} C. Berge: \emph{Hypergraphs: Combinatorics of Finite Sets}, North-Holland, Amsterdam, 1989.
\bibitem{Bermond} J. C. Bermond, M. C. Heydemann, and D. Sotteau: Line graphs of hypergraphs I, Discrete Mathematics {\bf 18} no. 3 (1977), 235--241.
\bibitem{GaoPeng} J. Gao and J. Peng: Counting shellings of complete bipartite graphs and trees, arXiv:1809.10263
\bibitem{Pa18} S. Palcoux: Number of collinear ways to fill a grid, MathOverflow, https://mathoverflow.net/questions/297385/number-of-collinear-ways-to-fill-a-grid
\bibitem{Sch19} Problems of the 2019 Mikl\'os Schweitzer Memorial Competition in Mathematics, http://www.math.u-szeged.hu/$\sim$mmaroti/schweitzer/schweitzer-2019-eng.pdf

\bibitem{Stanley} R. Stanley: Counting ``connected'' edge orderings (shellings) of the complete graph. MathOverflow, https://mathoverflow.net/questions/2974

\bibitem{Tu77} P. Tur\'an: An unusual life, Ramanujan, I and II, \emph{K\"oz\'episk. Mat. Lapok} {\bf 55} (1977), 49--54 and 97--106 [in
Hungarian]. Also: Ein sonderbarer Lebensweg, Ramanujan, in: \emph{Grosse Augenblicke aus der Geschichte der Mathematik (R. Freud, Hrsg.)}, B.I. Wissenschaftsverlag, Mannheim, 1990.
\bibitem{Ziegler} G. Ziegler: \emph{Lectures on Polytopes}, Graduate texts in Mathematics \textbf{152}, Springer, Berlin, 1995 (Chapter 8).

\end{thebibliography}
\end{document}